\documentclass[leqno,12pt]{amsart}
\usepackage{latexsym}
\usepackage{amssymb,amsmath,amsfonts}
\usepackage{mathtools}
\usepackage{mathrsfs}
\usepackage[utf8]{inputenc}
\usepackage[usenames,dvipsnames]{color}

\def\R{\mathbb{R}}

\def\N{\mathbb{N}}

\newtheorem{theorem}{Theorem}[section]
\newtheorem{lemma}[theorem]{Lemma}
\newtheorem{proposition}[theorem]{Proposition}

\theoremstyle{definition}
\newtheorem{definition}[theorem]{Definition}

\theoremstyle{remark}
\newtheorem{remark}[theorem]{Remark}
\numberwithin{equation}{section}

\newcommand{\E}{\mathfrak{E}}
\newcommand{\C}{\mathscr{C}}

\newcommand{\Fix}{\operatorname{Fix}}
\newcommand{\Id}{\operatorname{Id}}

\usepackage[left=2.cm,right=2.cm,top=1.8cm,bottom=1.8cm]{geometry}

\makeatletter
\def\section{\@startsection{section}{1}%
  \z@{.7\linespacing\@plus\linespacing}{.5\linespacing}%
  {\normalfont\bfseries\centering}}

\def\@setauthors{%
  \begingroup
  \trivlist
  \centering
  \@topsep 22\p@\relax
  \advance\@topsep by -\baselineskip
  \item\relax
  {\normalfont\normalsize\authors\par}
  \vspace{3\p@}
  {\normalfont\small
   School of Mathematical and Statistical Sciences, Northern Illinois University, DeKalb, IL 60115, USA\par}
  \vspace{1\p@}
  {\normalfont\small\texttt{markjoeuba@gmail.com}\par}
  \endtrivlist
  \endgroup
}
\makeatother

\begin{document}
\title[Perturbation-resilient inertial hybrid retractions]
{Perturbation-resilient inertial Krasnosel'ski\u\i-type hybrid retractions for generalized nonexpansive mappings}

\author[Markjoe O. Uba]{Markjoe O. Uba}


\begin{abstract}
Let $\E$ be a uniformly smooth and uniformly convex real Banach space.
We study an inertial hybrid retraction method for a countable sequence of
mappings satisfying the NST-condition and an approximate $\phi$-Fej\'er
inequality with vanishing errors. We prove that the generated sequence
converges strongly to the sunny generalized nonexpansive retraction of the initial point onto the common fixed-point set. The theorem admits vanishing error sequences that need not be summable and therefore contains the summable-error setting as a special case. We also establish a
Bregman-projection analogue and provide illustrative examples.
\end{abstract}

\maketitle
\enlargethispage{2pt}

\section{Introduction}

Inertial-type algorithms, originating in the work of Polyak \cite{Polyak1964},
are two-step iterative procedures in which the next iterate depends not only on
the current point but also on the preceding one. Such inertial terms are widely
used because they improve the practical speed of convergence while
preserving the structure of the underlying fixed point or optimization method.
Alongside inertial methods, hybrid projection and retraction techniques have
become important tools for obtaining strong convergence, especially in settings
where standard Krasnosel'ski\u\i--Mann iterations may only yield weak convergence.
Representative developments include Alber's generalized projection and
retraction framework, the hybrid method of Takahashi, Takeuchi, and Kubota,
and the error-free inertial hybrid scheme of Chidume and Nnakwe
\cite{Alber1996,TakahashiTakeuchiKubota2008,ChidumeNnakwe2020}.

In smooth Banach spaces, the normalized duality mapping \(J\) allows one to
replace the Hilbert-space squared norm geometry by a Banach-space geometry
generated by Alber's functional
\begin{equation}\label{eq:intro-phi}
\phi(x,y):=\|x\|^{2}-2\langle x,Jy\rangle+\|y\|^{2},
\qquad x,y\in \E.
\end{equation}
This functional plays the role of a generalized squared distance and is central
in the study of sunny generalized nonexpansive retractions and generalized
nonexpansive mappings; see, for instance,
\cite{Alber1996,IbarakiTakahashi2010,KohsakaTakahashi2007}. If \(T\)
has a nonempty fixed point set, then the generalized nonexpansive condition may
be expressed as the Fej\'er-type inequality
\begin{equation}\label{eq:fejer-exact}
\phi(Tx,p)\leq \phi(x,p),
\qquad x\in \E,\quad p\in \Fix(T).
\end{equation}
Thus \eqref{eq:fejer-exact} gives a natural decrease condition relative to the
fixed point set.

Most strong convergence results for hybrid retraction schemes are formulated in
an exact setting: the operator values are assumed to be evaluated without error,
and the Fej\'er inequality is assumed to hold exactly. In practical
implementations, however, exact evaluation may not be available. Approximation
errors may arise from approximate operator evaluations and related computational
inaccuracies. These considerations lead naturally to the question of whether
strong convergence can be obtained under controlled inexactness.

The aim of this paper is to prove strong convergence for a scheme with a
vanishing relaxation of the exact \(\phi\)-Fej\'er decrease condition. More
precisely, we allow the Fej\'er estimate to hold with a nonnegative error term \(\varepsilon_k\), 
\(\varepsilon_k\to0\), without requiring the summability condition
\(\sum_{k=1}^{\infty}\varepsilon_k<\infty\), and show that the inertial hybrid
retraction sequence still converges strongly to \(R_{F(\Gamma)}v_0\).

The Bregman framework studied in this paper has a long history. Bregman's 1967
article introduced a generalized projection and relaxation method for finding a
common point of convex sets \cite{Bregman1967}. Subsequent developments by many
authors substantially expanded this framework; see, for example,
\cite{CensorLent1981,CensorZenios1992,BauschkeBorwein1997,
ButnariuIusem2000,KohsakaTakahashi2005} and the references therein.

Against this background, the second purpose of the paper is to formulate and
prove a Bregman-projection analogue within the broader theory of Bregman
distances and projections.

\medskip
\noindent{\bf Related Work.}
Further developments in strong convergence methods for maximal monotone
problems and hybrid fixed-point schemes include
\cite{BelloAMUC2022,ChidumeOtuboEzeaUba2017,ChidumeUbaUzochukwuOtuboIdu2019,
IbarakiTakahashi2010,KlineamSuantaiTakahashi2012,
UbaOtuboOnyido2021,UbaCarpathian2023}.
Error-tolerant and perturbation-resilient iterative methods have also been
studied in fixed point theory and optimization; see, for example,
\cite{Combettes2001,IbarakiSaejung2023,ReichSabach2010}.
Recent work on inertial and perturbed Krasnosel'ski\u\i--Mann-type schemes includes
\cite{CortildPeypouquet2025}. Additional foundations for Bregman geometry and
projection methods may be found in \cite{ButnariuResmerita2006}.

\section{Preliminaries}

Let $\E$ be a real Banach space with dual $\E^*$. The normalized duality mapping $J:\E\to2^{\E^*}$ is defined by
\[
 Jx=\bigl\{x^*\in\E^*: \langle x,x^*\rangle=\|x\|^2=\|x^*\|^2\bigr\}.
\]
If $\E$ is smooth, then $J$ is single-valued. If $\E$ is smooth, strictly convex, and reflexive, then $J$ is a bijection from $\E$ onto $\E^*$; see \cite{Alber1996}. Throughout Sections~2 and 3, $\phi:\E\times\E\to \mathbb{R}$ is the Lyapunov functional defined by \eqref{eq:intro-phi}, and satisfies
\begin{equation}\label{eq:phi-lower}
 (\|x\|-\|y\|)^2\leq \phi(x,y)\leq(\|x\|+\|y\|)^2.
\end{equation}

\begin{lemma}[\cite{KamimuraTakahashi2003}]\label{lem:phi-consistency}
Let $\E$ be smooth and uniformly convex. Suppose that $\{x_n\}$ and $\{y_n\}$ are sequences in $\E$, at least one of which is bounded. If $\phi(x_n,y_n)\to0$, then $\|x_n-y_n\|\to0$.
\end{lemma}

A retraction $R_D:\E\to D$ is \emph{sunny} if
\[
 R_D\bigl(R_Dx+t(x-R_Dx)\bigr)=R_Dx,
\]
whenever $x\in\E$, $t\geq0$, and the displayed argument belongs to the domain. It is \emph{generalized nonexpansive} if
\[
 \phi(R_Dx,R_Dy)\leq \phi(x,y),
 \qquad x,y\in\E.
\]

\begin{lemma}[\cite{IbarakiTakahashi2010,KohsakaTakahashi2007}]\label{lem:retraction-facts}
Let $\E$ be smooth and strictly convex, let $D\subset\E$ be a nonempty closed sunny generalized nonexpansive retract, and let $R_D$ be its sunny generalized nonexpansive retraction. Then
\begin{equation}\label{eq:retraction-pythagorean}
 \phi(x,R_Dx)+\phi(R_Dx,z)\leq\phi(x,z),
 \qquad x\in\E,\ z\in D.
\end{equation}
Moreover, $R_Dx$ is the unique minimizer of $z\mapsto\phi(x,z)$ over $D$.
\end{lemma}

\begin{lemma}[\cite{KohsakaTakahashi2007}]\label{lem:retract-characterization}
Let $\E$ be smooth, strictly convex, and reflexive, and let $D\subset\E$ be nonempty and closed. Then $D$ is a sunny generalized nonexpansive retract of $\E$ if and only if $JD$ is closed and convex in $\E^*$.
\end{lemma}

Following Ibaraki and Takahashi \cite{IbarakiTakahashi2010}, a mapping $T:\E\to\E$ is called \emph{generalized nonexpansive} if $\Fix(T)\neq\varnothing$ and
\[
 \phi(Tx,p)\leq\phi(x,p),
 \qquad x\in\E,\ p\in\Fix(T).
\]
For a family $\Gamma$ of mappings, write
\[
 F(\Gamma):=\bigcap_{T\in\Gamma}\Fix(T).
\]
The following consequence of fixed-set results for generalized nonexpansive mappings will be useful.

\begin{lemma}\label{lem:common-retract}
Let $\E$ be smooth, strictly convex, and reflexive. Let $\Gamma$ be a family of generalized nonexpansive mappings on $\E$ such that $F(\Gamma)\neq\varnothing$. Then $F(\Gamma)$ is a sunny generalized nonexpansive retract of $\E$.
\end{lemma}

\begin{proof}
For each $T\in\Gamma$, the set $\Fix(T)$ is closed and $J\Fix(T)$ is closed and convex; see Ibaraki and Takahashi \cite{IbarakiTakahashi2010}. Since $J$ is one-to-one,
\[
 JF(\Gamma)=\bigcap_{T\in\Gamma}J\Fix(T).
\]
Indeed, let $x^*\in\bigcap_{T\in\Gamma}J\Fix(T)$. For each $T\in\Gamma$, choose $p_T\in\Fix(T)$ with $Jp_T=x^*$. Since $J$ is one-to-one, all the points $p_T$ coincide with a single point $p\in F(\Gamma)$, and hence $x^*=Jp\in JF(\Gamma)$. Thus $JF(\Gamma)$ is closed and convex. The set $F(\Gamma)$ is closed, and Lemma~\ref{lem:retract-characterization} applies.
\end{proof}

\begin{definition}[NST-condition; Klin-eam--Suantai--Takahashi \cite{KlineamSuantaiTakahashi2012}]\label{def:nst}
Let $\{T_k\}_{k\geq1}$ and $\Gamma$ be families of mappings on $\E$ such that
\[
 \bigcap_{k=1}^{\infty}\Fix(T_k)=F(\Gamma)\neq\varnothing.
\]
The sequence $\{T_k\}$ satisfies the \emph{NST-condition with $\Gamma$} if, for every bounded sequence $\{x_k\}\subset\E$,
\[
 \|x_k-T_kx_k\|\to0
 \quad\Longrightarrow\quad
 \|x_k-Tx_k\|\to0
 \quad\text{for every }T\in\Gamma.
\]
\end{definition}

\section{Inertial Hybrid Retraction Theorem}

\begin{definition}\label{def:phi-error}
Let $F\subset\E$ be nonempty. A sequence $\{T_k\}_{k\geq1}$ of mappings $T_k:\E\to\E$ has the \emph{vanishing-error $\phi$-Fej\'er property with respect to $F$} if there is a sequence $\{\varepsilon_k\}\subset[0,\infty)$ such that $\varepsilon_k\to0$ and
\begin{equation}\label{eq:phi-error}
 \phi(T_kx,p)\leq\phi(x,p)+\varepsilon_k,
 \qquad x\in\E,\ p\in F,\ k\geq1.
\end{equation}
In particular, whenever \eqref{eq:phi-error} holds with a nonnegative
summable error sequence, the requirement $\varepsilon_k\to0$ is automatically
satisfied.
\end{definition}

\begin{theorem}\label{thm:main}
Let $\E$ be a uniformly smooth and uniformly convex real Banach space. Let $\{T_k\}_{k\geq1}$ be mappings from $\E$ into itself, and let $\Gamma$ be a family of closed generalized nonexpansive mappings on $\E$ such that
\[
 \bigcap_{k=1}^{\infty}\Fix(T_k)=F(\Gamma)\neq\varnothing.
\]
Assume that $\{T_k\}$ satisfies the NST-condition with $\Gamma$ and the vanishing-error $\phi$-Fej\'er property with respect to $F(\Gamma)$. Let $0\leq\alpha_k\leq\overline\alpha<1$ for every $k$, and let $\{\beta_k\}\subset\R$ be bounded. Given $v_0\in\E$, set $v_1=v_0$ and $\C_1=\E$. For $k\geq1$, define
\begin{equation}\label{eq:main-algorithm}
\begin{aligned}
 w_k&=v_k+\beta_k(v_k-v_{k-1}),\\
 y_k&=\alpha_k w_k+(1-\alpha_k)T_kw_k,\\
 \C_{k+1}
 &=\bigl\{v\in\C_k:
 \phi(y_k,v)\leq\phi(w_k,v)+(1-\alpha_k)\varepsilon_k\bigr\},\\
 v_{k+1}&=R_{\C_{k+1}}v_0.
\end{aligned}
\end{equation}
Then every $\C_k$ is a nonempty closed sunny generalized nonexpansive retract of $\E$, so the iteration is well defined. Moreover,
\[
 v_k\longrightarrow R_{F(\Gamma)}v_0
\]
strongly.
\end{theorem}

\begin{proof}
Set $F:=F(\Gamma)$. By Lemma~\ref{lem:common-retract}, $F$ is a sunny generalized nonexpansive retract.

\smallskip
\noindent\emph{Step 1: Construction of the Shrinking Retracts.}
We prove inductively that
\begin{equation}\label{eq:induction-properties}
 F\subset\C_k,\qquad \C_k\text{ is closed},\qquad J\C_k\text{ is closed and convex}.
\end{equation}
The assertions hold for $\C_1=\E$. Indeed, uniform smoothness and uniform convexity imply that $\E$ is smooth, strictly convex, and reflexive. Hence $J:\E\to\E^*$ is bijective and $J\C_1=J\E=\E^*$.

Assume \eqref{eq:induction-properties} at index $k$. Put
\[
 \delta_k=(1-\alpha_k)\varepsilon_k,
 \qquad
 c_k=\frac{\|w_k\|^2-\|y_k\|^2+\delta_k}{2}.
\]
Expanding $\phi$ shows that, for $v\in\C_k$,
\[
 \phi(y_k,v)\leq\phi(w_k,v)+\delta_k
 \quad\Longleftrightarrow\quad
 \langle w_k-y_k,Jv\rangle\leq c_k.
\]
Thus
\begin{equation}\label{eq:primal-cut}
 \C_{k+1}=\{v\in\C_k:\langle w_k-y_k,Jv\rangle\leq c_k\}.
\end{equation}
Uniform smoothness implies that $J$ is norm-to-norm continuous on bounded subsets. To verify closedness, let $v_n\in\C_{k+1}$ and suppose that $v_n\to v$. Then $v\in\C_k$ because $\C_k$ is closed. The sequence $\{v_n\}$ is bounded, so $Jv_n\to Jv$. Consequently, continuity of the duality pairing yields
\[
 \langle w_k-y_k,Jv\rangle
 =\lim_{n\to\infty}\langle w_k-y_k,Jv_n\rangle\leq c_k.
\]
Hence $v\in\C_{k+1}$, and $\C_{k+1}$ is closed.

Let $p\in F$. By the induction hypothesis, $p\in\C_k$. Since $x\mapsto\phi(x,p)$ is convex, \eqref{eq:phi-error} gives
\begin{align*}
 \phi(y_k,p)
 &\leq \alpha_k\phi(w_k,p)+(1-\alpha_k)\phi(T_kw_k,p)\\
 &\leq \phi(w_k,p)+(1-\alpha_k)\varepsilon_k.
\end{align*}
Hence $p\in\C_{k+1}$, and therefore $F\subset\C_{k+1}$.

Define the dual halfspace
\[
 H_k^*=\{x^*\in\E^*:\langle w_k-y_k,x^*\rangle\leq c_k\}.
\]
Because $J$ is bijective, \eqref{eq:primal-cut} yields the exact identity
\[
 J\C_{k+1}=J\C_k\cap H_k^*.
\]
The right-hand side is closed and convex. Lemma~\ref{lem:retract-characterization} therefore shows that $\C_{k+1}$ is a sunny generalized nonexpansive retract.

\smallskip
\noindent\emph{Step 2: Boundedness and the Nested-Set Estimate.}
Fix $p\in F$. Since $p\in\C_k$ and $v_k=R_{\C_k}v_0$, Lemma~\ref{lem:retraction-facts} gives
\begin{equation}\label{eq:main-bounded}
 \phi(v_0,v_k)\leq\phi(v_0,p),
 \qquad k\geq1.
\end{equation}
By \eqref{eq:phi-lower}, $\{v_k\}$ is bounded. Since $\C_{k+1}\subset\C_k$ and $v_{k+1}\in\C_k$, applying \eqref{eq:retraction-pythagorean} to $R_{\C_k}v_0=v_k$ with $z=v_{k+1}$ gives
\[
 \phi(v_0,v_k)+\phi(v_k,v_{k+1})\leq\phi(v_0,v_{k+1}).
\]
In particular, $\phi(v_0,v_k)\leq\phi(v_0,v_{k+1})$.
Thus $\{\phi(v_0,v_k)\}$ is nondecreasing and bounded. Let
\[
 \ell=\lim_{k\to\infty}\phi(v_0,v_k).
\]
If $m>n$, then $v_m\in\C_m\subset\C_n$, and \eqref{eq:retraction-pythagorean} gives
\begin{equation}\label{eq:main-nested-estimate}
 \phi(v_n,v_m)
 \leq\phi(v_0,v_m)-\phi(v_0,v_n).
\end{equation}

\smallskip
\noindent\emph{Step 3: Convergence and Residual Decay.}
We now prove that $\{v_k\}$ is Cauchy. If not, there would exist $\rho>0$ and indices $m_j>n_j\to\infty$ such that $\|v_{m_j}-v_{n_j}\|\geq\rho$ for all $j$. Since $\phi(v_0,v_k)\to\ell$, \eqref{eq:main-nested-estimate} gives
\[
 0\leq\phi(v_{n_j},v_{m_j})
 \leq\phi(v_0,v_{m_j})-\phi(v_0,v_{n_j})\longrightarrow0.
\]
The sequence $\{v_k\}$ is bounded, so Lemma~\ref{lem:phi-consistency} yields $\|v_{m_j}-v_{n_j}\|\to0$, a contradiction. Thus $\{v_k\}$ is Cauchy. Since $\E$ is complete,
\begin{equation}\label{eq:vk-to-q}
 v_k\to q,
\end{equation}
for some $q\in\E$. For each fixed $n$, one has $v_k\in\C_k\subset\C_n$ whenever $k\geq n$. Since $\C_n$ is closed, $q\in\C_n$. Hence $q\in\bigcap_{n\geq1}\C_n$.

Let $B=\sup_k|\beta_k|<\infty$. From \eqref{eq:vk-to-q},
\[
 \|w_k-v_k\|\leq B\|v_k-v_{k-1}\|\to0,
 \qquad\text{so }w_k\to q.
\]
Because $v_{k+1}\in\C_{k+1}$,
\[
 0\leq\phi(y_k,v_{k+1})
 \leq\phi(w_k,v_{k+1})+(1-\alpha_k)\varepsilon_k.
\]
The convergences $w_k\to q$ and $v_{k+1}\to q$, together with norm-to-norm continuity of $J$ on bounded sets, imply $\phi(w_k,v_{k+1})\to\phi(q,q)=0$. Also, $0\leq(1-\alpha_k)\varepsilon_k\leq\varepsilon_k\to0$. Thus the right-hand side tends to zero. Lemma~\ref{lem:phi-consistency} gives
\[
 \|y_k-v_{k+1}\|\to0.
\]
Moreover,
\[
 \|v_{k+1}-w_k\|
 \leq \|v_{k+1}-q\|+\|w_k-q\|\to0.
\]
Hence $\|y_k-w_k\|\to0$. Since
\[
 y_k-w_k=(1-\alpha_k)(T_kw_k-w_k),
\]
and $1-\alpha_k\geq1-\overline\alpha>0$, we obtain
\begin{equation}\label{eq:main-residual}
 \|T_kw_k-w_k\|\to0.
\end{equation}

\smallskip
\noindent\emph{Step 4: Application of the NST-Condition and $q\in F$.}
The sequence $\{w_k\}$ is bounded. By the NST-condition and \eqref{eq:main-residual},
\[
 \|Tw_k-w_k\|\to0,
 \qquad T\in\Gamma.
\]
Fix $T\in\Gamma$. Since $w_k\to q$ and $\|Tw_k-w_k\|\to0$, we have $Tw_k\to q$. Therefore $(w_k,Tw_k)\to(q,q)$, and closedness of the graph of $T$ yields $Tq=q$. Hence $q\in F$.

\smallskip
\noindent\emph{Step 5: Identification of the Limit.}
Let $p\in F$. Since $v_k\to q$ and $Jv_k\to Jq$ by the
norm-to-norm continuity of $J$ on bounded sets, one has
\[
 \phi(v_0,v_k)\to\phi(v_0,q).
\]
Thus \eqref{eq:main-bounded} implies
\[
 \phi(v_0,q)=\lim_{k\to\infty}\phi(v_0,v_k)\leq\phi(v_0,p).
\]
Thus $q$ minimizes $z\mapsto\phi(v_0,z)$ over $F$. Lemma~\ref{lem:retraction-facts} gives $q=R_Fv_0$.
\end{proof}

\section{A Bregman-Projection Analogue}\label{sec:bregman}

Let $f:\E\to\R$ be convex and G\^ateaux differentiable. Its Bregman distance is
\[
 D_f(x,y)=f(x)-f(y)-\langle x-y,\nabla f(y)\rangle.
\]
We use the standard projection orientation
\[
 P_D^f x=\operatorname*{argmin}_{z\in D}D_f(z,x).
\]

\begin{definition}\label{def:admissible-bregman}
Let $\E$ be reflexive. A function $f:\E\to\R$ is called an \emph{admissible Bregman generator} if
\begin{enumerate}
 \item $f$ is Legendre and strongly coercive;
 \item $f$ is bounded and uniformly Fr\'echet differentiable on bounded subsets of $\E$;
 \item $f$ is uniformly convex on bounded subsets of $\E$;
 \item for every $x\in\E$ and $r>0$, the right Bregman sublevel set
 \[
  \{y\in\E:D_f(x,y)\leq r\},
 \]
 is bounded.
\end{enumerate}
\end{definition}

The following lemma gives consequences of Legendre duality,
Bregman projection theory, and total convexity; see
\cite{AlberButnariu1997,BauschkeBorwein1997,ButnariuIusem2000,
ButnariuIusemZalinescu2003,ButnariuResmerita2006}.

\begin{lemma}\label{lem:bregman-facts}
Let $f$ be an admissible Bregman generator. Then:
\begin{enumerate}
 \item $\nabla f:\E\to\E^*$ is a bijection,
 \[
  \nabla f^*=(\nabla f)^{-1}.
 \]
 Both gradients are bounded and uniformly norm-to-norm continuous on
 bounded subsets of their domains; see
 \cite{BauschkeBorwein1997,ButnariuIusem2000}.

 \item If $D\subset\E$ is nonempty, closed, and convex, then the Bregman
 projection
 \[
  P_D^f x=\operatorname*{argmin}_{z\in D}D_f(z,x),
 \]
 exists uniquely and satisfies
 \[
  D_f(z,P_D^f x)+D_f(P_D^f x,x)\leq D_f(z,x),
  \qquad z\in D.
 \]
 see
 \cite{AlberButnariu1997,BauschkeBorwein1997,
 ButnariuResmerita2006}.

 \item For bounded sequences $\{x_n\}$ and $\{y_n\}$,
 \begin{equation}\label{eq:bregman-consistency}
  D_f(x_n,y_n)\to0
  \quad\Longleftrightarrow\quad
  \|x_n-y_n\|\to0.
 \end{equation}
 The implication from vanishing Bregman distance to norm convergence
 follows from total convexity on bounded sets; see
 \cite{ButnariuIusemZalinescu2003,ButnariuResmerita2006}.
 The converse follows from the uniform Fr\'echet differentiability of
 $f$ on bounded subsets of $\E$.
\end{enumerate}
\end{lemma}
For $x\in\E$ and $x^*\in\E^*$, define
\[
 V_f(x,x^*)=f(x)-\langle x,x^*\rangle+f^*(x^*).
\]
Then
\begin{equation}\label{eq:V-identity}
 V_f(x,x^*)=D_f\bigl(x,\nabla f^*(x^*)\bigr),
\end{equation}
and $x^*\mapsto V_f(x,x^*)$ is convex; see, for example,
\cite{BauschkeBorwein1997,ButnariuIusem2000}.

\begin{definition}\label{def:bregman-error}
Let $F\subset\E$ be nonempty. A sequence $\{T_k\}_{k\geq1}$ has the \emph{vanishing-error $D_f$-Fej\'er property with respect to $F$} if there is a sequence $\{\varepsilon_k\}\subset[0,\infty)$ with $\varepsilon_k\to0$ such that
\begin{equation}\label{eq:bregman-error}
 D_f(p,T_kx)\leq D_f(p,x)+\varepsilon_k,
 \qquad x\in\E,\ p\in F,\ k\geq1.
\end{equation}
In particular, whenever \eqref{eq:bregman-error} holds with a nonnegative summable error sequence, the requirement $\varepsilon_k\to0$ is automatically satisfied.
\end{definition}

Related shrinking-projection results for Bregman asymptotically
quasi-nonexpansive mappings in the intermediate sense were obtained in
\cite{Tomizawa2014}, and an inertial hybrid Bregman scheme for a
variational-like system and a single asymptotic fixed-point mapping was studied
in \cite{AlNemerAliFarid2025}.  The result below differs in that the NST-condition carries the residual convergence from $\{T_k\}$ to every mapping in the countable family, while the additive errors are required only to vanish.

\begin{theorem}\label{thm:bregman}
Let $\E$ be a reflexive real Banach space and let $f$ be an admissible Bregman generator. Let $\{T_k\}_{k\geq1}$ be mappings from $\E$ into itself, and let $\Gamma$ be a family of closed mappings such that
\[
 \bigcap_{k=1}^{\infty}\Fix(T_k)=F(\Gamma)\neq\varnothing.
\]
Assume that $F(\Gamma)$ is closed and convex, that $\{T_k\}$ satisfies the NST-condition with $\Gamma$, and that \eqref{eq:bregman-error} holds with $F=F(\Gamma)$. Let $0\leq\alpha_k\leq\overline\alpha<1$ for every $k$, and let $\{\beta_k\}\subset\R$ be bounded. Given $v_0\in\E$, set $v_1=v_0$ and $\C_1=\E$. Define
\begin{equation}\label{eq:bregman-algorithm}
\begin{aligned}
 w_k&=v_k+\beta_k(v_k-v_{k-1}),\\
 y_k&=\nabla f^*\!\left(\alpha_k\nabla f(w_k)
 +(1-\alpha_k)\nabla f(T_kw_k)\right),\\
 \C_{k+1}
 &=\bigl\{v\in\C_k:
 D_f(v,y_k)\leq D_f(v,w_k)+(1-\alpha_k)\varepsilon_k\bigr\},\\
 v_{k+1}&=P_{\C_{k+1}}^f v_0.
\end{aligned}
\end{equation}
Then every $\C_k$ is nonempty, closed, and convex, the iteration is well defined, and
\[
 v_k\longrightarrow P_{F(\Gamma)}^f v_0
\]
strongly.
\end{theorem}

\begin{proof}
Set $F:=F(\Gamma)$ and $\delta_k=(1-\alpha_k)\varepsilon_k$.

\smallskip
\noindent\emph{Step 1: The Sets $\C_k$ Are Closed and Convex and Contain $F$.}
The assertion holds at $k=1$ because $\C_1=\E$ is nonempty, closed, and convex, and $F\subset\C_1$. Suppose inductively that $\C_k$ is nonempty, closed, and convex and that $F\subset\C_k$. For $v\in\C_k$, expansion of the two Bregman distances gives
\begin{align*}
 D_f(v,y_k)-D_f(v,w_k)
 &=\langle v,\nabla f(w_k)-\nabla f(y_k)\rangle\\
 &\quad+f(w_k)-f(y_k)
 +\langle y_k,\nabla f(y_k)\rangle
 -\langle w_k,\nabla f(w_k)\rangle.
\end{align*}
Hence, setting
\[
 d_k=\delta_k+f(y_k)-f(w_k)
 +\langle w_k,\nabla f(w_k)\rangle
 -\langle y_k,\nabla f(y_k)\rangle,
\]
we obtain
\begin{equation}\label{eq:bregman-halfspace}
 \C_{k+1}
 =\{v\in\C_k:\langle v,\nabla f(w_k)-\nabla f(y_k)\rangle\leq d_k\}.
\end{equation}
Thus $\C_{k+1}$ is closed and convex.

Since $\nabla f^*=(\nabla f)^{-1}$, the definition of $y_k$ gives
\[
 \nabla f(y_k)
 =\alpha_k\nabla f(w_k)+(1-\alpha_k)\nabla f(T_kw_k).
\]
Let $p\in F$. Using \eqref{eq:V-identity}, convexity of $V_f(p,\cdot)$, and \eqref{eq:bregman-error},
\begin{align*}
 D_f(p,y_k)
 &\leq \alpha_kD_f(p,w_k)+(1-\alpha_k)D_f(p,T_kw_k)\\
 &\leq D_f(p,w_k)+\delta_k.
\end{align*}
Therefore $p\in\C_{k+1}$. Induction shows that $\C_k$ is nonempty, closed, and convex for every $k\geq1$.

\smallskip
\noindent\emph{Step 2: Boundedness and the Nested-Set Estimate.}
Fix $p\in F$. The Bregman projection inequality gives
\begin{equation}\label{eq:bregman-bounded}
 D_f(p,v_k)+D_f(v_k,v_0)\leq D_f(p,v_0).
\end{equation}
The map $z\mapsto D_f(z,v_0)=f(z)-\langle z,\nabla f(v_0)\rangle+f^*(\nabla f(v_0))$ is coercive. Hence the boundedness of $\{D_f(v_k,v_0)\}$ in \eqref{eq:bregman-bounded} implies that $\{v_k\}$ is bounded. Since $v_{k+1}\in\C_{k+1}\subset\C_k$,
\[
 D_f(v_{k+1},v_k)+D_f(v_k,v_0)\leq D_f(v_{k+1},v_0).
\]
Thus $\{D_f(v_k,v_0)\}$ is nondecreasing and bounded. If $m>n$, then $v_m\in\C_n$, and
\begin{equation}\label{eq:bregman-nested}
 D_f(v_m,v_n)
 \leq D_f(v_m,v_0)-D_f(v_n,v_0).
\end{equation}

\smallskip
\noindent\emph{Step 3: Convergence and Residual Decay.}
Let $a_k=D_f(v_k,v_0)$. The preceding argument shows that $a_k$ converges. If $\{v_k\}$ were not Cauchy, there would exist $\rho>0$ and indices $m_j>n_j\to\infty$ such that $\|v_{m_j}-v_{n_j}\|\geq\rho$. By \eqref{eq:bregman-nested},
\[
 0\leq D_f(v_{m_j},v_{n_j})\leq a_{m_j}-a_{n_j}\longrightarrow0.
\]
Both subsequences are bounded, so \eqref{eq:bregman-consistency} gives $\|v_{m_j}-v_{n_j}\|\to0$, a contradiction. Thus $\{v_k\}$ is Cauchy. Let $v_k\to q$. For each fixed $n$, $v_k\in\C_k\subset\C_n$ for all $k\geq n$. Since $\C_n$ is closed, $q\in\C_n$, and hence $q\in\bigcap_n\C_n$. Boundedness of $\{\beta_k\}$ yields $w_k\to q$.

We next show that the auxiliary sequences are bounded. Since $\{w_k\}$ is bounded and both $f$ and $\nabla f$ are bounded on bounded sets, the identity
\[
 D_f(p,w_k)=f(p)-f(w_k)-\langle p-w_k,\nabla f(w_k)\rangle,
\]
shows that $\{D_f(p,w_k)\}$ is bounded. Equation \eqref{eq:bregman-error}, together with boundedness of the convergent sequence $\{\varepsilon_k\}$, then shows that $\{D_f(p,T_kw_k)\}$ is bounded. The right-sublevel assumption in Definition~\ref{def:admissible-bregman} implies that $\{T_kw_k\}$ is bounded. Consequently, $\{\nabla f(w_k)\}$ and $\{\nabla f(T_kw_k)\}$ are bounded in $\E^*$. Their convex combinations are bounded, and boundedness of $\nabla f^*$ on bounded subsets of $\E^*$ gives boundedness of $\{y_k\}$.

Since $v_{k+1}-w_k\to0$, \eqref{eq:bregman-consistency} gives $D_f(v_{k+1},w_k)\to0$. Since $v_{k+1}\in\C_{k+1}$, the inequality defining $\C_{k+1}$ in \eqref{eq:bregman-algorithm} gives
\[
 0\leq D_f(v_{k+1},y_k)
 \leq D_f(v_{k+1},w_k)+\delta_k\to0.
\]
Hence $\|v_{k+1}-y_k\|\to0$, and therefore $\|y_k-w_k\|\to0$. Uniform continuity of $\nabla f$ on bounded sets yields
\[
 \|\nabla f(y_k)-\nabla f(w_k)\|\to0.
\]
By the definition of $y_k$,
\[
 \nabla f(y_k)-\nabla f(w_k)
 =(1-\alpha_k)\bigl(\nabla f(T_kw_k)-\nabla f(w_k)\bigr).
\]
Since $1-\alpha_k\geq1-\overline\alpha>0$, it follows that
\[
 \|\nabla f(T_kw_k)-\nabla f(w_k)\|\to0.
\]
The two gradient sequences are bounded. Hence, by the uniform norm-to-norm continuity of $\nabla f^*$ on bounded subsets of $\E^*$,
\begin{align*}
 \|T_kw_k-w_k\|
 &=\bigl\|\nabla f^*(\nabla f(T_kw_k))
 -\nabla f^*(\nabla f(w_k))\bigr\|\longrightarrow0.
\end{align*}
Thus
\begin{equation}\label{eq:bregman-residual}
 \|T_kw_k-w_k\|\to0.
\end{equation}

\smallskip
\noindent\emph{Step 4: Membership in $F$.}
By the NST-condition and \eqref{eq:bregman-residual},
\[
 \|Tw_k-w_k\|\to0,
 \qquad T\in\Gamma.
\]
Since $w_k\to q$ and $\|Tw_k-w_k\|\to0$, one has $Tw_k\to q$. Closedness of the graph of $T$ therefore gives $Tq=q$. Thus $q\in F$.

\smallskip
\noindent\emph{Step 5: Identification of the Limit.}
From \eqref{eq:bregman-bounded},
\[
 D_f(v_k,v_0)\leq D_f(p,v_0),
 \qquad p\in F.
\]
Since $z\mapsto D_f(z,v_0)$ is continuous and $v_k\to q$,
\[
 D_f(q,v_0)=\lim_{k\to\infty}D_f(v_k,v_0)\leq D_f(p,v_0),\qquad p\in F.
\]
Thus $q$ is the unique minimizer of $z\mapsto D_f(z,v_0)$ over $F$, and hence $q=P_F^fv_0$.
\end{proof}

\begin{remark}[Why the Two Theorems Use Different Defining Inequalities]\label{rem:orientation}
For $\psi(x)=\frac12\|x\|^2$ on a smooth Banach space,
\[
 D_\psi(x,y)=\frac12\phi(x,y).
\]
Nevertheless, Theorem~\ref{thm:main} minimizes $D_\psi(v_0,\cdot)$,
whereas the standard Bregman projection minimizes $D_f(\cdot,v_0)$. The
asymmetry is structural. In particular,
\[
 D_f(a,v)-D_f(b,v)=f(a)-f(b)+\langle b-a,\nabla f(v)\rangle,
\]
so the set obtained by placing the variable in the second argument need not be convex. The defining inequality for $\C_{k+1}$ in \eqref{eq:bregman-algorithm} is affine in $v$, as shown by \eqref{eq:bregman-halfspace}.
\end{remark}

\section{Illustrative Examples}

This section illustrates the two convergence theorems in three settings. The first example constructs a nonzero vanishing-error family on the real line to which Theorem~\ref{thm:main} applies. The second verifies the assumptions of Definition~\ref{def:admissible-bregman} for a power-type Legendre generator. The third uses the corresponding Bregman projections to construct an averaged operator covered by Theorem~\ref{thm:bregman}. Power-type Legendre generators and their Bregman projections are standard in Bregman convex analysis and feasibility theory; see \cite{BauschkeBorwein1997,ButnariuIusem2000,ButnariuIusemZalinescu2003}. Averaged Bregman projection constructions are classical in convex feasibility; see \cite{ButnariuCensorReich1997,Naraghirad2020}.

\subsection{A Nonzero Vanishing-Error Family on the Real Line}

\begin{proposition}\label{prop:scalar-example}
Let $\E=\R$, let $\eta_k>0$ be pairwise distinct with $\eta_k\to0$, and define
\[
 Sx=\frac{x}{1+|x|},
 \qquad
 T_kx=(1+\eta_k)Sx.
\]
Set $\Gamma=\{S\}$. Then the operator-family assumptions of
Theorem~\ref{thm:main} are satisfied with
\[
 F(\Gamma)=\{0\},
 \qquad
 \varepsilon_k=2\eta_k+\eta_k^2.
\]
Consequently, for every choice of parameters satisfying
\[
 0\leq\alpha_k\leq\alpha<1
 \quad\text{and}\quad
 \sup_k|\beta_k|<\infty,
\]
the sequence generated by \eqref{eq:main-algorithm} converges strongly to $0$.
\end{proposition}

\begin{proof}
Here $J=\Id$ and $\phi(x,y)=|x-y|^2$. The mapping $S$ is continuous,
$\Fix(S)=\{0\}$, and $|Sx|\leq|x|$, so $S$ is generalized nonexpansive.
Solving $T_kx=x$ gives
\[
 \Fix(T_k)=\{0,-\eta_k,\eta_k\}.
\]
The pairwise distinctness of $\eta_k$ therefore yields
\[
 \bigcap_{k=1}^{\infty}\Fix(T_k)=\{0\}=\Fix(S).
\]
Moreover,
\[
 |T_kx-Sx|=\eta_k|Sx|\leq\eta_k.
\]
Thus, for every sequence $\{x_k\}$,
\[
 |x_k-Sx_k|\leq|x_k-T_kx_k|+\eta_k,
\]
which proves the NST-condition with $\Gamma$.

Finally, $|Sx|\leq1$ and $|Sx|\leq|x|$, so
\begin{align*}
 |T_kx|^2
 &=(1+\eta_k)^2|Sx|^2\\
 &=|Sx|^2+(2\eta_k+\eta_k^2)|Sx|^2\\
 &\leq |x|^2+(2\eta_k+\eta_k^2).
\end{align*}
This is \eqref{eq:phi-error} with $F=\{0\}$. Since
$\varepsilon_k\to0$, Theorem~\ref{thm:main} applies. In particular, choosing
$\eta_k=1/k$ gives
\[
 \varepsilon_k=\frac{2}{k}+\frac{1}{k^2}\longrightarrow0,
 \qquad
 \sum_{k=1}^{\infty}\varepsilon_k=\infty.
\]
\end{proof}

\subsection{A Power-Type Bregman Generator}

Let $d,m\in\N$, let $1<r<\infty$, and set
$s=r/(r-1)$. Equip $\R^d$ with the $\ell^r$ norm and its standard duality
with $\ell^s$. Define
\begin{equation}\label{eq:power-generator}
 f_r(x)=\frac1r\sum_{j=1}^{d}|x_j|^r.
\end{equation}
Then
\[
 (\nabla f_r(x))_j=|x_j|^{r-2}x_j,
 \qquad
 (\nabla f_r^*(x^*))_j=|x_j^*|^{s-2}x_j^*.
\]

\begin{lemma}\label{lem:power-generator}
The function $f_r$ in \eqref{eq:power-generator} is an admissible Bregman
generator.
\end{lemma}

\begin{proof}
The Legendre, strong-coercivity, smoothness, and uniform-convexity properties
are standard for the finite-dimensional power function; see
\cite{BauschkeBorwein1997,ButnariuIusem2000,
ButnariuIusemZalinescu2003}. It remains only to give the right-sublevel
condition used in Definition~\ref{def:admissible-bregman}. Fenchel's identity
gives
\[
 D_{f_r}(x,y)=\frac1r\|x\|_r^r+\frac1s\|y\|_r^r
 -\langle x,\nabla f_r(y)\rangle.
\]
Since $\|\nabla f_r(y)\|_s=\|y\|_r^{r-1}$, H\"older's inequality yields
\[
 D_{f_r}(x,y)\geq
 \frac1s\|y\|_r^r-\|x\|_r\|y\|_r^{r-1}
 +\frac1r\|x\|_r^r.
\]
For fixed $x$, the right-hand side tends to $+\infty$ as
$\|y\|_r\to\infty$. Hence every right Bregman sublevel set
$\{y:D_{f_r}(x,y)\leq a\}$ is bounded, and $f_r$ is admissible.
\end{proof}

\subsection{An Averaged Bregman Projection Operator}

Let
$\C_1,\ldots,\C_m\subset\R^d$ be nonempty closed convex sets with
\[
 \C:=\bigcap_{i=1}^{m}\C_i\neq\varnothing.
\]
Let $\lambda_i>0$ and $\sum_{i=1}^{m}\lambda_i=1$. Write
$P_i=P_{\C_i}^{f_r}$ and define
\[
 Tx=\nabla f_r^*\!\left(\sum_{i=1}^{m}\lambda_i\nabla f_r(P_i x)\right).
\]

\begin{proposition}\label{prop:averaged-bregman}
The mapping $T$ is continuous,
\[
 \Fix(T)=\C,
\]
and, for every $p\in\C$,
\begin{equation}\label{eq:strict-bregman-fejer}
 D_{f_r}(p,Tx)
 \leq D_{f_r}(p,x)-\sum_{i=1}^{m}\lambda_iD_{f_r}(P_i x,x).
\end{equation}
Consequently, Theorem~\ref{thm:bregman} applies with $T_k=T$,
$\Gamma=\{T\}$, and $\varepsilon_k=0$, and its iterates converge to
$P_{\C}^{f_r}v_0$.
\end{proposition}

\begin{proof}
We first verify continuity. Fix $i$, let $x_n\to x$, and set
$z_n=P_i x_n$. Choose $a\in\C_i$. By the minimizing property of $P_i$,
\[
 D_{f_r}(z_n,x_n)\leq D_{f_r}(a,x_n).
\]
The right-hand side is bounded. Moreover, with
$M:=\sup_n\|\nabla f_r(x_n)\|_s<\infty$, one has
\[
 D_{f_r}(z,x_n)
 \geq \frac1r\|z\|_r^r-M\|z\|_r,
\]
so the left-hand sublevel sets are bounded uniformly in $n$. Hence
$\{z_n\}$ is bounded. If a subsequence $z_{n_j}\to z$, then
$z\in\C_i$, and continuity of $D_{f_r}$ gives, for every $y\in\C_i$,
\[
 D_{f_r}(z,x)
 =\lim_{j\to\infty}D_{f_r}(z_{n_j},x_{n_j})
 \leq\lim_{j\to\infty}D_{f_r}(y,x_{n_j})
 =D_{f_r}(y,x).
\]
Uniqueness of the Bregman projection yields $z=P_i x$. Thus every
subsequence has a further subsequence converging to $P_i x$, and therefore
$P_i x_n\to P_i x$. Hence every $P_i$, and consequently $T$, is continuous.

Let $p\in\C$. By convexity of $V_{f_r}(p,\cdot)$ and the Bregman
projection inequality,
\begin{align*}
 D_{f_r}(p,Tx)
 &\leq\sum_{i=1}^m\lambda_iD_{f_r}(p,P_i x)\\
 &\leq D_{f_r}(p,x)
 -\sum_{i=1}^m\lambda_iD_{f_r}(P_i x,x),
\end{align*}
which proves \eqref{eq:strict-bregman-fejer}. If $x\in\C$, then
$P_i x=x$ for every $i$, and hence $Tx=x$. Conversely, if $Tx=x$, then
\eqref{eq:strict-bregman-fejer} gives
\[
 \sum_{i=1}^m\lambda_iD_{f_r}(P_i x,x)=0.
\]
Every term is nonnegative and every $\lambda_i$ is positive, so
$D_{f_r}(P_i x,x)=0$ for all $i$. Strict convexity of $f_r$ implies
$P_i x=x$, and therefore $x\in\C$. Thus $\Fix(T)=\C$.

Since $T$ is continuous, it is closed. For the constant sequence $T_k=T$,
the NST-condition with $\Gamma=\{T\}$ is immediate, and
\eqref{eq:strict-bregman-fejer} implies the vanishing-error condition with
$\varepsilon_k=0$. Therefore Theorem~\ref{thm:bregman} applies.
\end{proof}

\section{Conclusions}

We established strong convergence of an inertial shrinking-retraction method for a countable family of mappings satisfying the NST-condition and a vanishing-error $\phi$-Fej\'er inequality, without requiring the error sequence to be summable. We further obtained a Bregman-projection analogue and illustrated both results with explicit examples. In the two settings, the generated sequence converges strongly to the retraction and the Bregman projection, respectively, of the initial point onto the common fixed-point set.

\end{document}